\documentclass{svproc}
\usepackage[
	english,					
]{babel} 					
\usepackage[
	activate,				
	kerning=true,			
	factor=1100,				
	final,					
	spacing=true,			
	stretch=10,				
	shrink=10,				
	tracking=true,			
]{microtype}	
\usepackage[
	intlimits,				
]{amsmath} 					
\usepackage{amssymb}			
\usepackage[
	makeroom					
]{cancel}  					
\usepackage{courier}			
\usepackage[
	scaled=0.95,				
]{helvet}					
\usepackage{icomma}			
\usepackage{mathtools}		
\usepackage{nicefrac}		
\usepackage[
	short,					
]{optidef}					
\usepackage{upgreek}			
\usepackage[
	algochapter,				
	ruled,					
]{algorithm2e}				
\usepackage{blindtext}		
\usepackage[
	font={small},			
	labelfont=bf,			
	textformat=period,		
]{caption}					
\usepackage{enumitem}		
\usepackage{footnote} 		
\usepackage[
	framemethod=TikZ,		
]{mdframed} 					
\usepackage[
	all						
]{nowidow}					
\usepackage[
	above,					
	below,					
]{placeins}					
\usepackage[
	onehalfspacing			
]{setspace} 					
\usepackage{subcaption}		
\usepackage{xcolor}			
\usepackage{svg} 
\usepackage{float} 
\newtheorem{theorem1}{Theorem}
\newtheorem{remark1}[theorem1]{Remark}

\usepackage[a-2b]{pdfx}		
\newcommand{\R}{\mathbb{R}}
\newcommand{\uF}{u_h}

\makeatletter
\def\mathcenterto#1#2{\mathclap{\phantom{#1}\mathclap{#2}}\phantom{#1}}
\let\old@widetilde\widetilde
\def\td#1#2{\mathcenterto{#2}{\old@widetilde{\mathcenterto{#1}{#2\,}}}}
\let\old@widehat\widehat
\def\widehatto#1#2{\mathcenterto{#2}{\old@widehat{\mathcenterto{#1}{#2\,}}}}
\makeatother

\usepackage[backend=bibtex]{biblatex}
\usepackage{url}

\begin{document}
\mainmatter              
\title{An Eikonal Approach for Globally Optimal Free Flight Trajectories}
\titlerunning{Globally Optimal Free Flight Trajectories}  
%
\author{Ralf Borndörfer\inst{1,2} \and Arturas Jocas\inst{1,2} \and Martin Weiser\inst{1}}
\authorrunning{Ralf Borndörfer et al.} 
%
%
\institute{Zuse Institute Berlin, Takustr. 7, 14195 Berlin, Germany\\
\and
Freie Universität Berlin, Arnimallee 14, 14195 Berlin, Germany
}

\maketitle              

\begin{abstract}

We present an eikonal-based approach that is capable of finding a continuous globally optimal trajectory for an aircraft in a stationary wind field. This minimizes emissions and fuel consumption. If the destination is close to a cut locus of the associated Hamilton-Jacobi-Bellman equation, small numerical discretization errors can lead to selecting a merely locally optimal trajectory and missing the globally optimal one. Based on finite element error estimates, we construct a trust region around the cut loci in order to guarantee uniqueness of trajectories for destinations sufficiently far from cut loci.

\end{abstract}
\section{Introduction}

The aviation sector is expanding annually, the market size in 2025 being nearly 359 billion
USD \cite{mordor_intelligence}. Every day, approximately 2.8 million people commute
within Europe using airplanes \cite{Eurostat2025} and 13 million worldwide  \cite{iata_passengers_2024}. This requires about 99 billion gallons of jet fuel per year \cite{iata_outlook_2024}, which accounts for roughly 10\% of the energy consumption of the transportation sector \cite{irena_aviation_transport}, which in turn has a share of 30\% of the total worldwide energy consumption, and the same holds for the associated emissions.  

The fuel consumption and hence the emissions of flights are essentially proportional to their duration, which depends on the flight trajectory. In spatially varying wind fields, the fastest trajectory is in general not the great circle arc connecting origin and destination. Indeed, airlines routinely plan aircraft routes in order to take the effect of tail- and headwinds into account, asking for globally optimal trajectories with respect to shortest flight duration. This is traditionally done on a three-dimensional airway network consisting of 400,000 waypoints connected by 900,000 arcs on each of 43 horizontal flight levels by Dijkstra-type methods~\cite{Blanco2023}, taking a number of additional flight restrictions into account~\cite{RAD}. The computation of horizontal and vertical profiles is typically decoupled~\cite{NgSridharGrabbe2014}.

We are concerned in this paper with improving the horizontal trajectories from routes on a discrete airway network to arbitrary so-called free flight trajectories~\cite{RTCA1995}. The mathematical challenge in extending the design space is to solve a continuous free flight optimization problem to global optimality. This is difficult because of the possible existence of multiple local minimizing trajectories, see Fig.~\ref{fig:cut_locus_example}. 

\begin{figure}
    \centering        
    \includegraphics[width=0.8\textwidth, trim={1cm 0.5cm 1cm 1.4cm}, clip]{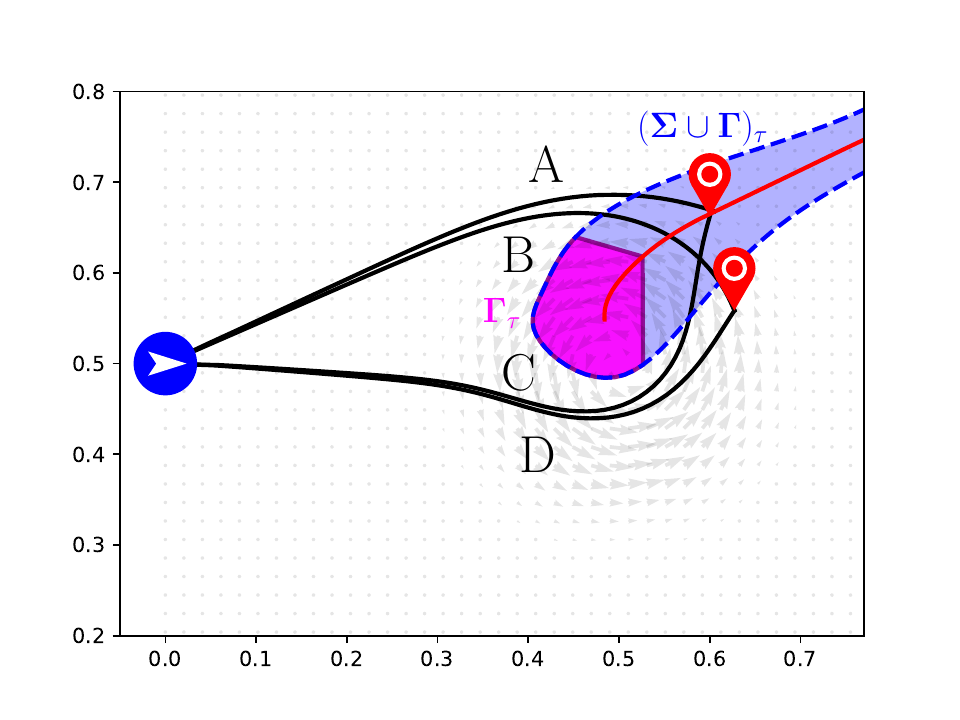}
    \caption{Optimal trajectories from the blue origin to two different red destination points in a single-vortex wind field are indicated as solid black lines; A, C, D are globally optimal, B is only locally optimal. Estimated singularities, so called cut loci, which can be reached by multiple globally optimal trajectories, are shown as a red line. The blue region indicates a (scaled) trust region around the cut loci; the magenta region depicts the region around the set of conjugate points, see Def.~\ref{def:nearly_conj_pt} below}
    \label{fig:cut_locus_example}
\end{figure}
Depending on the considered scenario and the associated model, various approaches have been discussed in the literature, see, e.g., \cite{Danecker_thesis} for a comprehensive survey.
We propose in this paper the use of dynamic programming~\cite{Bellman1957} methods for solving the Hamilton-Jacobi-Bellman (HJB) equation corresponding to the horizontal free flight optimization problem. The first algorithms of this type appeared in the 1990s. Tsitsiklis and later Sethian proposed Dijkstra-type, level set, and Fast Marching Methods (FMM)~\cite{Tsitsiklis1995,FMM, FMM2} focusing on the isotropic eikonal equation. However, free flight optimization problem is not isotropic due to the directional dependence of the wind, and a generalized eikonal equation representation is required. The proposed solution, Sethian's Ordered Upwind Method (OUM)~\cite{SethianOUM_first}, is effectively an extension of FMMs to generalized eikonal equations utilizing semi-Lagrangian (Hopf-Lax) or Eulerian discretizations. Girardet et al.~\cite{GirardetEtAl2014} have applied such semi-Lagrangian OUMs to find optimal trajectories in a planar free flight setting. Another approach proposed by Bornemann and Rasch~\cite{Bornemann2006, Rasch2007} uses a linear finite element (FE) discretization which is equivalent to the semi-Lagrangian discretization. The key difference is how the resulting system of equations is solved: instead of a Dijkstra-like approach, iterative approaches resembling Jacobi and Gauss-Seidel iterations are proposed. Wells et al.~\cite{WellsEtAl2023} have used a Jacobi-type iteration in the study of optimizing horizontal flight profiles across the North Atlantic, showing potential air-distance reductions of 4.2\% on average. However, as far as the authors of this paper are aware, there has been no attempt to establish rigorous global optimality guarantees for such approaches yet. In this paper, we also solve the HJB equation corresponding to the planar free flight setting using linear finite elements and a Hopf-Lax update scheme. We perform an investigation of the singularities of the HJB equation which paves the way for proving the global optimality of such an approach.

The paper is organized as follows. In Sec.~\ref{sec:eikonal}, we introduce the Free Flight Trajectory Optimization Problem with the corresponding HJB equation and provide a singularity analysis based on the results of Pignotti~\cite{PIGNOTTI2002681}. We construct an arbitrarily accurate approximation of the computational domain on which the value function is sufficiently regular, which allows to bound the discretization error by an arbitrarily small quantity $\varepsilon$~\cite{Rasch2007}. We then define a temporal trust region of size $2 \varepsilon$ around the singularities of the HJB solution. If the flight destination is outside of this arbitrarily small region, the numerical HJB solution provides a provably good approximation of the unique globally optimal trajectory. Finally, Sec.~\ref{sec:results} gives computational results verifying that the discretization error estimate is linear for various wind fields. In addition, we verify that outside of the aforementioned temporal trust region the destinations indeed have a unique globally optimal trajectory.

\section{An Eikonal Approach} \label{sec:eikonal}

In this section we model the horizontal Free Flight Trajectory Optimization Problem (FFTOP) on a domain $\mathbb{R}^2$. We consider a constant airspeed $\overline{v} \in \mathbb{R}$, a stationary, bounded, and continuously differentiable wind field $w \in BC^1(\mathbb{R}^2,\mathbb{R}^2$), and an objective of minimizing the travel time. Let an aircraft start at an origin $x_0 \in \mathbb{R}^2$ at time $t = 0$ and move with constant airspeed $\overline{v}$ until it reaches a destination point $x_d  \in \mathbb{R}^2$. It follows a trajectory $y$ governed by the equations of motion
\begin{equation}
\left\{
    \begin{aligned}
    y_t(t) &= f(y(t), a(t)) \, a(t), \quad t \in [0, T_a], \\
    y(0) &= x_0, \; y(T_a) = x_d,\label{chapter2:eq_motion}
    \end{aligned}
\right.
\end{equation}
where $a : [0, T_a] \to S_1$ is a measurable directional control function with control set $S_1 = \{d \in \mathbb{R}^2 \mid \lVert d \rVert = 1 \}$. Note that the arrival or terminal time $T_a$ depends on the chosen control. This is the so-called \textit{Zermelo problem} proposed in 1930s~\cite{Zermelo1930}. Zermelo's work outlines the necessary conditions for an optimal control in the 2D case. In our case, the ground speed in direction $p$ is 
\begin{equation}
    f(x, p) = \sqrt{\overline{v}^2 - \lVert w(x) \rVert^2 + ( w(x)^T p)^2 } + w(x) ^T p \quad \forall x \in \R^2, p \in S_1  ,
\end{equation}
where we assume $\lVert w \rVert < \overline{v}$ everywhere. The FFTOP is then defined as
\begin{equation}  
\min\limits_{a \in \mathcal{A}} T_a \quad \text{ s.t. } 
\left\{
    \begin{aligned}
    &y_t(t) = f(y(t), a(t)) \, a(t), \quad t \in [0, T_a], \\
    &y(0) = x_0, \; y(T_a) = x_d,\\
    &\mathcal{A} = \{a : [0, T_a] \to S_1 \mid a \text{ is measurable}\}.
    \end{aligned}
\right.
\label{chapter2:FFTOP}
\end{equation}
Note that a fixed control $a$ induces a unique trajectory $y^{x_0,a}$, however, the uniqueness of the minimizing control $a$ itself is not guaranteed. Fig.~\ref{fig:cut_locus_example} shows an example of a destination reached by two distinct globally optimal trajectories (A and C). It also shows a destination reached by one globally and one locally optimal trajectory (B and D) which illustrates that the existence of multiple local minimizers makes solving FFTOP~\eqref{chapter2:FFTOP} to global optimality difficult with locally convergent methods. 

In contrast, solving the corresponding Hamilton-Jacobi-Bellman (HJB) equation provides global information on the trajectories, see, e.g.,~\cite{evans10}. For a fixed control ${a}$ and origin $x_0$, we denote the minimum arrival time at some arbitrary point $x$ by
\begin{equation}\label{eq:arrival-time}
    {\tau}(x, {a}) = \inf \{ t \geq 0 \mid {y}^{x_0, {a}}(t) = x \},
\end{equation}
or ${\tau}(x, {a}) = \infty$, if $y^{x_0, a}(t) \neq x, \forall t \geq 0$. The minimum arrival time over all controls, given by the \textit{value function}, is 
\begin{equation} \label{chapter2:value_func}
u(x) =
\begin{aligned}
      &\min \{ {\tau}(x, {a}) \mid {a} \in \mathcal{A}  \},  &&x \in \R^2 .
\end{aligned}
\end{equation}
Utilizing Bellman's principle of optimality~\cite{Bellman1957}, one can show, see, e.g., \cite{cannarsa2004}, that $u$ is a weak solution of the stationary HJB equation
\begin{equation}
    \begin{aligned}
    H (x, u_x(x)) = 0&, \quad x \in \R^2 \setminus \{x_0\},\\
    u(x_0) = 0&,
    \end{aligned}  
    \label{chapter2:hjb_pde-1}
\end{equation}
where the Hamiltonian $H$ is defined as
\begin{equation}
    H(x, p) = -\min\limits_{a \in S_1} \left[ p^Ta f(x, -a) \right] - 1.
\end{equation}
The accuracy of numerical discretizations for~\eqref{chapter2:hjb_pde-1} often suffers from the intrinsic singularity at the origin $x_0$. One therefore considers a closed origin 
set $\mathcal{K}$ with $x_0 \in \text{int}(\mathcal{K})$, and the restricted HJB equation
\begin{equation}
    \begin{aligned}
     H (x, u_x(x)) = 0&, \quad x \in \R^2 \setminus \mathcal{K},\\
    u(x) = g(x)&, \quad x \in \mathcal{K}.
    \end{aligned}  
    \label{chapter2:hjb_pde}
\end{equation}
For the regularity it is sufficient that $\partial \mathcal{K}$ is of class ${C}^1$. In addition, we assume the existence of a boundary condition oracle $g$ such that $g = u \mid_{\mathcal{K}} $ holds.

\subsection{Regularity of Hamilton-Jacobi-Bellman Solutions}
\label{section:regularity}
Due to the non-uniqueness of minimizing controls for FFTOP, the corresponding HJB equation $\eqref{chapter2:hjb_pde}$ usually does not admit classical solutions. Therefore, an alternative notion of unique \textit{viscosity solutions} was introduced by Crandall \& Lions \cite{crandall1983}.
\begin{definition}
    A bounded and uniformly continuous function $u$ is the \textit{viscosity solution} of equation $\eqref{chapter2:hjb_pde}$ if the following conditions hold for each smooth test function $\phi \in C^\infty_c (\R^2, \mathbb{R})$:

(i) if $u - \phi$ has a local minimum at $x \in \R^2$, then
$
    H(x, \phi_x(x)) \geq 0,
$

(ii) if $u - \phi$ has a local maximum at $x \in \R^2$, then
$
    H(x, \phi_x(x)) \leq 0.
$
\label{eq:viscosity_soln}
\end{definition}
\noindent It can be shown that viscosity solutions are unique \cite{evans10}.

Our main result Thm.~$\ref{theorem:cutting_out_domain}$ is based on notation and results following from the regularity analysis of Pignotti~\cite{PIGNOTTI2002681}. She showed that FFTOP is equivalent to solving, for a fixed $\xi \in \partial \mathcal{K}$, the Hamiltonian system 
\begin{equation}
\begin{cases}
    y_t(t) = H_p(y(t), p(t))   \\
    p_t(t) = -H_x(y(t), p(t))
\end{cases}
\label{eq:HJB_pignotti}
\end{equation}
with initial conditions
\begin{equation}
\begin{cases}
    y(0) = \xi   \\
    p(0) = g_x(\xi) + \mu(\xi)\nu(\xi)
\end{cases}
\end{equation}
and solution $y(\xi, t)$, $p(\xi, t)$. Here, $\mu(\xi)>0$ is a unique constant, such that $H(\xi, g_x(\xi) + \mu(\xi)\nu(\xi)) = 0$ holds with $\nu(\xi)$ being the outer normal unit vector at $\xi \in \partial\mathcal{K}$. 
Note that, the value function $u$ is defined implicitly through $y(\xi, u(x)) = x$ for some $\xi \in \partial \mathcal{K}$, as well as $p(\xi, u(x)) = u_x(y(\xi, u(x)))$.

We define the \emph{singular set}  of $u$ as $\Sigma := \{x \in \R^2 \setminus \mathcal{K} \mid u \text{ not differentiable in } x\}$. 
Denote for each destination $x$ the set of optimal starting points on $\partial \mathcal{K}$ by $\mathcal{F}(x) := \{\xi \in \partial \mathcal{K} \mid y(\xi, u(x)) = x  \}$,
by $y_{\xi, t}(\cdot)$ and $p_{\xi, t}(\cdot)$ the Jacobian of $y$ and $p$ with respect to the pair $(\xi, t)$, and by $\Gamma := \{x\in\R^2\backslash \mathcal{K} \mid \exists \xi \in \mathcal{F}(x): \det y_{\xi, t}(\xi, u(x)) = 0\}$ the set of \emph{conjugate points}, where the derivative $y_\xi$ is understood to be only in tangential direction.

If one additionally assumes that $w \in C^3(\R^2, \mathbb{R}^2)$ and $\partial \mathcal{K}$ is of class ${C}^3$, the value function $u$ is as regular in the complement of $\Sigma \cup \Gamma$ as the speed function and the topology of the origin set allows it to be:

\begin{theorem} \cite[Theorem 3.1]{PIGNOTTI2002681}
Let $u$ be a solution of $\eqref{chapter2:hjb_pde}$. Denote the set of regular points as $\mathcal{R} = \R^2 \setminus \{ \mathcal{K} \cup \Sigma \cup \Gamma \}$. If $w \in C^3(\R^2, \mathbb{R}^2)$ and $\partial \mathcal{K}$ is $C^3$, then $\Sigma \cup \Gamma$ is closed and the value function $u \in C^2 (\mathcal{R} , \mathbb{R})$.    \label{theorem:smoothness_value_function}
\end{theorem}

\noindent This is a key result with respect to the global optimality of numerically computed
optimal trajectories for FFTOP. Note that Thm.~\ref{theorem:smoothness_value_function} holds for any domain $\Omega_0 \subset \R^2$, such that the solution of~\eqref{chapter2:hjb_pde} on $\Omega_0$ coincides with the respective solution of $\eqref{chapter2:hjb_pde}$ on $\R^2$. In this case we will say that $\Omega_0$ is a \emph{travel-time preserving} subset of $\R^2$.  Furthermore, $\Sigma \cup \Gamma$ is effectively a one-dimensional set embedded in $\R^2$, indeed, $\Sigma \cup \Gamma$ is $\mathcal{H}^1$-rectifiable~\cite{PIGNOTTI2002681}. For the rest of this paper, we will therefore assume that $\Sigma \cup \Gamma$ is a countable collection of Lipschitz-continuous curves. This is relevant for bounding the discretization error for FFTOP in Thm.~\ref{theorem:error_estimate}. Under the same assumptions as in Thm. $\ref{theorem:smoothness_value_function}$, the closure of the singular set coincides with the union of the singular and conjugate sets, i.e., $\overline{\Sigma} = \Sigma \cup \Gamma$~\cite{Mennucci2000}[Theorem 3.8]. 
    
Using the above results, we construct an arbitrarily accurate approximation on a travel-time preserving $\Omega_\tau$ with $u \in C^2(\Omega_\tau, \mathbb{R})$ retaining the same optimal solutions of $\eqref{chapter2:hjb_pde}$ on $\Omega_\tau$ and removing all singularities. In order to do so we first need to define an exclusion neighborhood around the conjugate points.

\begin{definition}[$\tau$-nearly Conjugate Points]
     Let $u$ be a solution of~\eqref{chapter2:hjb_pde} and $y(\cdot, \cdot)$ the corresponding optimal trajectories satisfying the Hamiltonian system~\eqref{eq:HJB_pignotti}. Denote $t_{\rm max} > 0$ as the maximum exit time, and the exit time function by $t(\xi) := \min\{t>0\mid (y(\xi, t) \in \Sigma \cup \Gamma) \lor t = t_{\rm max} \}$. For small $\tau > 0$ define $\Gamma_\tau:=\{y(\xi, t) \mid \exists\xi \in \partial\mathcal{K} : y(\xi, t(\xi)) \in \Gamma, t \geq t(\xi) - \tau\} \cup \{y(\xi, t) \mid \xi, \xi' \in \partial \mathcal{K} : y(\xi,t(\xi)) \in \Sigma, y(\xi', t(\xi')) \in \Gamma,  \lVert y(\xi,t(\xi)) - y(\xi', t(\xi'))\rVert \leq \tau \overline{v}, t \geq t(\xi) - \tau   \}$ 
     (see magenta area in Fig.~\ref{fig:cut_locus_example}). 
     If $x \in \Gamma_\tau$, then we will say that $x$ is $\tau$-nearly conjugate. 
     \label{def:nearly_conj_pt}
\end{definition}

\begin{theorem}[Regularity]
    Let $u$ be a solution to $\eqref{chapter2:hjb_pde}$ and $y(\cdot, \cdot)$ the corresponding optimal trajectories satisfying the Hamiltonian system $\eqref{eq:HJB_pignotti}$. 
    Let $\Gamma_\tau$ be a set of $\tau$-nearly conjugate points (Def. $\ref{def:nearly_conj_pt}$). Denote a region of regular points $\Omega_\tau := \{y(\xi, t) \mid \xi \in \partial \mathcal{K}, t < t(\xi)\} \setminus  \Gamma_\tau$. Then for any $\min\limits_{\xi \in \partial \mathcal{K}} t(\xi) > \tau > 0$, the solution to $\eqref{chapter2:hjb_pde}$ on $\Omega_\tau$ coincides with $u$, and $\sup_{x \in {\Omega}_\tau} \lVert u_{xx}(x)\rVert < \infty$.
    \begin{proof}
        First, notice that by construction $\Omega_\tau$ is open and travel-time preserving. Consider some $x \in \Omega_\tau$. By construction of $\Omega_\tau$, $x \notin \Sigma \cup \Gamma$. In addition, there exists $\xi \in \partial \mathcal{K}$, such that $y(\xi, \cdot)$ is the optimal trajectory to $x$ on the domain $\Omega_0$. Note that the same optimal trajectory is valid on $\Omega_\tau$ as well, again by construction of $\Omega_\tau$. Then, $y(\xi, \cdot)$ is a unique such trajectory on $\Omega_\tau$, otherwise, $x$ would have been singular on $\Omega_0$. It also cannot be conjugate on the new domain $\Omega_\tau$, since this would also imply it was conjugate on $\R^2$. Therefore, the solution to $\eqref{chapter2:hjb_pde}$ on $\Omega_\tau$ coincides with $u$. In addition, since $\overline{\Sigma} = \Sigma \cup \Gamma$ and $\Omega_\tau$ being travel-time preserving, Thm.~\ref{theorem:smoothness_value_function} yields $u \in C^2(\Omega_\tau, \mathbb{R})$. 

        Now we will show that $\sup_{x \in \Omega_\tau} \lVert u_{xx}(x) \rVert < \infty$. We argue by contradiction. Suppose that $\sup_{x \in {\Omega}_\tau} \lVert u_{xx}(x)\rVert = \infty$, that is, for some $\overline{x} \in \overline{\Omega}_\tau$ there exists a sequence of regular points $(x_i)_{i \in \mathbb{N}}$ with $\lVert u_{xx} (x_i) \rVert \to \infty$, ${x}_i \to \overline{x}$. We distinguish three cases:
        
        (i) $\overline{x} \notin \partial\Omega_\tau$. One can define a compact subset of $\mathrm{int}(\Omega_\tau)$, and since $u_{xx} \in C^2(\Omega_\tau)$, we obtain $\sup_{x \in {\Omega}_\tau} \lVert u_{xx}(x)\rVert < \infty$, which is a contradiction.

        (ii) $\overline{x} \in \Sigma \cap \partial \Omega_\tau$. Since $\overline{x}$ is not conjugate by construction, the number of its distinct optimal trajectories (or limiting gradients) is finite, i.e., $k:= |\mathcal{F}(\overline{x})| < \infty$ ~\cite{PIGNOTTI2002681}[Theorem 3.2]. Simultaneously, this means that the neighborhood around $x$ is locally split into $k$ regular regions with $\Sigma$ in between them. This implies that there exists a convergent subsequence of $(x_i)_{i \in \mathbb{N}}$ contained in exactly one of those regions. Denote such open region near $\overline{x}$ in this neighborhood as $Q$ with $\overline{x} \in \partial Q$, and the aforementioned subsequence as $(\overline{x}_i)_{i \in \mathbb{N}}$. Note that for each $\overline{x}_i \in Q$ there exist unique $\overline{\xi}_i, \overline{t}_i$, such that, $y(\overline{\xi}_i, \overline{t}_i) = \overline{x}_i$. Moreover, $y$ is a bijection between $y^{-1}[Q]$ and $Q$, and since $y \in C^2$, $\lim_{i \to \infty} y(\overline{\xi}_i, \overline{t}_i) = y(\overline{\xi}, u(\overline{x}))$ with $\overline{\xi}_i \to \overline{\xi}$. Here $\overline{\xi} \in \partial \mathcal{K}$ is such that $p(\overline{\xi},u(\overline{x}))$ is the limiting gradient of $\overline{x}$ on $Q$.       Therefore, $\lVert p_{\xi, t}(\overline{\xi}_i, \overline{t_i})\rVert \to \lVert p_{\xi, t}(\overline{\xi}, u(\overline{x}))\rVert$ and $\lVert [y_{\xi, t}(\overline{\xi}_i, \overline{t_i})]^{-1}\rVert \to \lVert [y_{\xi, t}(\overline{\xi}, u(\overline{x}))]^{-1}\rVert$.

        Notice that $p(\overline{\xi}_i, \overline{t}_i) = u_x(y(\overline{\xi}_i, \overline{t}_i))$ holds for all $(\overline{\xi}_i, \overline{t}_i) \in y^{-1}[Q]$. As $y, p \in C^2$, differentiating both sides of it yields
        \begin{equation}
            u_{xx}(\overline{x}_i) = p_{\xi, t}(\overline{\xi}_i, \overline{t}_i) [y_{\xi, t}(\overline{\xi}_i, \overline{t}_i)]^{-1}, \quad (\overline{\xi}_i, \overline{t}_i) \in y^{-1}[Q].
            \label{eq:hessian_explicit}
        \end{equation}
        \noindent For a fixed $y(0) = \overline{\xi}$, Jacobians $\overline{y} = [y_{\xi, t}, p_{\xi, t}]^T$ satisfy the following system \cite{PIGNOTTI2002681}:
        \begin{align}
            \overline{y}_t(t) = \overline{H}(y(t), p(t))  \overline{y}(t), \quad \text{where} \quad\overline{H} = \begin{bmatrix}
                H_{px} \quad H_{pp} \\
                -H_{xx} \quad -H_{xp} \\
            \end{bmatrix}.
            \label{eq:jacobians_pignotti}
        \end{align}
        \noindent Using Gr\"onwall's lemma for $\eqref{eq:jacobians_pignotti}$ yields $ \lVert \overline{y}(u(\overline{x})) \rVert < \infty$ since $H \in C^3$, which in turn implies $\lVert y_{\xi, t}(\overline{\xi}, u(\overline{x})) \rVert < \infty$ and $\lVert p_{\xi, t}(\overline{\xi}, u(\overline{x})) \rVert < \infty$. Since $\overline{x} \notin \Gamma$, $\det (y_{\xi, t}(\overline{\xi}, u(\overline{x}))) \neq 0$. Now, $\partial\Omega_\tau \cap \Sigma$ is closed, and so is its pre-image under $y_{\xi, t}$, on which there exists a minimum of $|\mathrm{det}(y_{\xi,t}(\cdot, \cdot))|$ bounding it away from zero. Therefore, $\lVert [y_{\xi, t}(\overline{\xi}, u(\overline{x}))]^{-1}\rVert = \lVert y_{\xi, t}(\overline{\xi}, u(\overline{x}))\rVert / \det (y_{\xi, t}(\overline{\xi}, u(\overline{x}))) < \infty$. Finally, using the original assumption the LHS of $\eqref{eq:hessian_explicit}$ diverges to $\infty$, while the RHS is bounded using the above arguments implying a contradiction.

        (iii) $\overline{x} \in \partial \Omega_\tau \setminus \Sigma$. Now, $\overline{x}$ is regular with a single regular neighborhood around $\overline{x}$, and an equivalent argument as for $\overline{x} \in \Sigma \cap \partial \Omega_\tau$ follows.
        \end{proof}
    \label{theorem:cutting_out_domain}
\end{theorem}

\subsection{Discretization}
Now we will focus on the discretization of the HJB equation $\eqref{chapter2:hjb_pde}$. Let $\mathcal{T}_h$ be a regular triangulation of $\overline{\Omega}_\tau$, covering $\overline{\Omega}_\tau$, such that $T\cap \Omega_\tau \ne \emptyset$ for all $T\in\mathcal{T}$, with $h$ as the maximum diameter and $h_\perp$ as the minimum altitude of  triangles $T\in \mathcal{T}_h$, and let $\theta = h/h_\perp$ be the mesh aspect ratio.
The finite element space of continuous and piecewise linear functions is $V_h=\{ v\in C(\overline{\Omega}_\tau,\mathbb{R})\mid v|_T \in \mathbb{P}_1\}$. $v_h\in V_h$ is uniquely determined by its values $v_h(x_h)$ on the nodes $x_h\in\mathcal{N}_h$ of $\mathcal{T}_h$. Let $\Omega_h := \mathcal{N}_h\cap (\overline{\Omega}_\tau \setminus\mathcal{K})$ and $\partial \Omega_h := \mathcal{N}_h\cap \partial\mathcal{K}$. The Hopf-Lax update $\Lambda_h:V_h\to V_h$ is defined as
    \begin{align}
        (\Lambda_h v_h)(x_h) = \begin{cases}
            \min\limits_{y \in \partial \omega_h(x_h)} \left (   
            v_h(y) + \frac{\lVert x_h - y\rVert}{f \left(x_h, \frac{
            x_h - y}{\lVert x_h - y\rVert} \right)}
            \right), \quad  &x_h \in \Omega_h, \\
            g(x_h), \quad &x_h \in \partial \Omega_h,
        \end{cases}
        \label{eq:FE_update_formula}
    \end{align}
where $\omega_h(x_h) = \bigcup_{T\in\mathcal{T}_h: x_h\in T} T$ is the patch around node $x_h$. 

The FE solution $\uF$ is the unique fixed point of $\Lambda_h$, i.e., $\Lambda_h \uF = \uF$, following the notation in Bornemann \& Rasch \cite{Bornemann2006}. 
Note that a necessary requirement for their analysis to hold is $\overline{\Omega}_\tau$ having a locally Lipschitz boundary. In our case it is satisfied due to sufficient regularity of the initial conditions on $\eqref{eq:HJB_pignotti}$, and $\Sigma \cup \Gamma$ being a collection of Lipschitz continuous curves. In \cite{Bornemann2006} one can also find the existence and uniqueness proofs as well as a proof of convergence to the viscosity solution. An explicit linear error bound is obtained in regions where $u$ is $C^2$ and a weaker result of order $\sqrt{h}$ for the general case~ \cite{Rasch2007}. Following the proof of Rasch~\cite{Rasch2007}, plugging in and estimating some constants adapted to our case, we extend the discretization error estimate, such that it applies for all $x \in \overline{\Omega}_\tau$, i.e., everywhere but on $\Gamma_\tau$. 

\begin{theorem}[Discretization Error
\label{theorem:error_estimate}
\cite{Rasch2007}]
Let $u$ be a solution to $\eqref{chapter2:hjb_pde}$ on $\overline{\Omega}_\tau$ with a locally Lipschitz boundary. Using Thm.~\ref{theorem:cutting_out_domain}, denote $\mathcal{H}^\star := \sup\limits_{x \in \Omega_\tau} \lVert u_{xx}(x) \rVert$. Let $\overline{c}_0 := \sup\limits_{x \in \R^2} \lVert w(x) \rVert$, $\overline{c}_1 = \sup\limits_{x \in \R^2} \lVert w_x(x) \rVert$, and $I_h$ be the FE nodal interpolator. Then, for all sufficiently small $h>0$, the following holds:
\begin{align}
     \varepsilon(h) &:=\max\limits_{x \in \overline{\Omega}_\tau} |u_h(x) - u(x) | \\
     &\leq~h \theta \left ( \mathcal{H}^\star + 2 \overline{c}_1 \sqrt{\overline{v}}  \frac{(\overline{v} + \overline{c}_0)^2}{(\overline{v} - \overline{c}_0)^{9/2}} \right )  (\overline{v} + \overline{c}_0) \lVert I_hu \rVert_\infty + \frac{7h}{\overline{v} - \overline{c}_0}. \label{eq:disc-error-bound}
\end{align}
    \begin{proof}
    First, note that for any $x \in \Omega_\tau$ in a simplex $T \in \mathcal{T}_h$ there exists $x_h \in \Omega_h \cap T$, such that:
    \begin{equation}
        |u_h(x) - u(x)| \leq |u_h(x_h) - u(x_h)| + \frac{h}{\overline{v} - \overline{c}_0}. \label{eq:thm2:1}
    \end{equation}
    Therefore, it is sufficient to focus on nodes $\Omega_h$. Denote $\Omega_\tau^h = \{x_h \in \Omega_h \mid \omega_h(x_h) \subset \Omega_\tau \}$. We distinguish two cases:
    
    (i) $x_h \in \Omega_\tau^h$. Define the slowness function as $\overline{f}(\cdot, \cdot) := 1/f(\cdot, \cdot)$. We show in Lem.~\ref{lemma:wind_constant} in the appendix that for any fixed $p \in S_1$,  $\overline{f}(\cdot, p)$ is Lipschitz continuous with a constant 
    \begin{equation}
    L = \overline{c}_1 \sqrt{\frac{\overline{v} }{(\overline{v} -          \overline{c}_0)^{5}} }
    \end{equation}
    \noindent that is independent of $p$. Thm.~2.15 from~\cite{Rasch2007} yields
\begin{align}
    \max\limits_{x_h \in \Omega_\tau^h} |u(x_h) - \Lambda_h I_h u (x_h) | &\leq h^2 \left ( \frac{1}{2} \mathcal{H}^\star + L \left ( \frac{\overline{v} + \overline{c}_0}{\overline{v} - \overline{c}_0} \right )^2 \right ) \\ 
    &\leq h^2 \left ( \frac{1}{2} \mathcal{H}^\star + \overline{c}_1 \sqrt {\frac{\overline{v}}{(\overline{v} - \overline{c}_0)^5}} \left ( \frac{\overline{v} + \overline{c}_0}{\overline{v} - \overline{c}_0} \right )^2 \right ) \\
    &=: h \mu (h),
\end{align}
where $\mu(h)$ is referred to as the \textit{modulus of continuity} of $I_hu$. Combining this result and Lem.~2.16 from~\cite{Rasch2007}, we obtain the error bound on $\Omega_\tau^h$:
\begin{align}
    &\hphantom{\leq}\max\limits_{x_h \in \Omega_\tau^h} |u_h(x_h) - u (x_h) | \\
    &=\max\limits_{x_h \in \Omega_\tau^h} |u_h(x_h) - I_h u (x_h) |\\ &\leq 2\theta \mu(h) (\overline{v} + \overline{c}_0) \lVert I_hu \rVert_\infty \\
    &\leq h \theta \left ( \mathcal{H}^\star + 2 \overline{c}_1 \sqrt {\frac{\overline{v}}{(\overline{v} - \overline{c}_0)^5}} \left ( \frac{\overline{v} + \overline{c}_0}{\overline{v} - \overline{c}_0} \right )^2 \right )  (\overline{v} + \overline{c}_0) \lVert I_hu \rVert_\infty. \label{eq:thm2:2}
\end{align}

(ii) $x_h \notin \Omega^h_\tau$. First we assume $\omega_h(x_h) \cap(\Sigma \cup \Gamma_\tau) \neq \emptyset$. For a small enough $h>0$, there exists $y_h \in \Omega_\tau^h$, such that, $\lVert x_h - y_h \rVert \leq 2h$ since $\Sigma \cup \Gamma$ is a collection of Lipschitz continuous curves. Using Thm.~8 from \cite{Bornemann2006} we get
\begin{align}
    |u_h(x_h) - u(x_h)| &\leq |u(y_h) - u(x_h)| +  |u_h(x_h) - u_h(y_h)| + |u_h(y_h) - u(y_h)|  \\
    &\leq \frac{\lVert x_h - y_h \rVert}{\overline{v} - \overline{c}_0} + 2\frac{\lVert x_h - y_h \rVert}{\overline{v}-\overline{c}_0} + |u_h(y_h) - u(y_h)|  \\
    &\leq \frac{6h}{\overline{v}-\overline{c}_0} + |u_h(y_h) - u(y_h)|. \label{eq:thm2:3}
\end{align}
\noindent In the other case $(\omega_h(x_h) \cap \partial \Omega_\tau) \setminus (\Sigma \cup \Gamma_\tau) \neq \emptyset$, an analogous argument as above holds, since $\Omega_\tau$ has a locally Lipschitz boundary. Finally, combining $\eqref{eq:thm2:1}$, $\eqref{eq:thm2:2}$, and $\eqref{eq:thm2:3}$ yields the claim~\eqref{eq:disc-error-bound}.
\end{proof}
\end{theorem}

\subsection{Localizing the Singularities}

Now we use the discretization error estimate~\eqref{eq:disc-error-bound} to localize the singularities $\Sigma \cup \Gamma$ of the HJB equation $\eqref{chapter2:hjb_pde}$. Instead of spatially localizing $\Sigma \cup \Gamma$, we work temporally using the framework introduced in Sec.~\ref{section:regularity}. Note that upon solving~\eqref{chapter2:hjb_pde}, one obtains not only arrival times, but also the predecessor directions. By considering the discrete divergence of the predecessor field (or simply by setting the threshold for the angle between the predecessor directions) we can identify a set of simplices $(\Sigma_h \cup \Gamma_h) \subset \mathcal{T}_h$ corresponding to the set $\Sigma \cup \Gamma$. 

\begin{theorem}[Cut Loci Position]
    Let $u$ be a solution to $\eqref{chapter2:hjb_pde}$ and $y(\cdot, \cdot)$ the corresponding optimal trajectory satisfying the Hamiltonian system~\eqref{eq:HJB_pignotti}. Denote by $(\Sigma \cup \Gamma)_\tau := \{y(\xi, t) \mid \exists\xi \in \partial\mathcal{K} : y(\xi, t(\xi)) \in \Sigma \cup \Gamma, t \geq t(\xi) - \tau\}$ the neighborhood around cut loci, see Fig. $\ref{fig:cut_locus_example}$. If $\varepsilon =\varepsilon(h)$ is the maximal discretization error restricted to $\overline{\Omega}_\tau$ following the notation from Thm. $\ref{theorem:error_estimate}$, and $\Sigma_h \cup \Gamma_h \subset \{y(\xi, t) \mid \exists\xi \in \partial\mathcal{K} : y(\xi, t(\xi)) \in \Sigma \cup \Gamma, t \leq t(\xi)\}$, then for a small enough $h>0$ we have $(\Sigma_h \cup \Gamma_h) \subset (\Sigma \cup \Gamma)_{2\varepsilon} $, that is, $\Sigma \cup \Gamma$ is contained in a $2 \varepsilon$ trust region around estimated $\Sigma_h \cup \Gamma_h$.
\begin{proof}
Let $\tau_0$ be the smallest time for which $(\Sigma \cup \Gamma)_{\tau_0}$ contains $\Sigma_h \cup \Gamma_h$, i.e.,
\begin{equation} \label{eq:tau0}
\tau_0 := \inf \{\tau > 0 \mid (\Sigma_h \cup \Gamma_h) \subset (\Sigma \cup \Gamma)_\tau\}.
\end{equation}

\noindent By construction of $(\Sigma \cup \Gamma)_{\tau_0}$, there exists at least one $A_h \in (\Sigma_h \cup \Gamma_h) \cap \partial ((\Sigma \cup \Gamma)_{\tau_0})$. Then there exists a unique optimal trajectory $y$ satisfying $\eqref{eq:HJB_pignotti}$, such that, $y(\xi, u(A_h)) = A_h$ for some $\xi \in \partial \mathcal{K}$. Following this optimal trajectory, we arrive at $A \in \Sigma \cup \Gamma$ with $y(\xi, u(A)) = A$ and $\tau_0 = u(A)  - u(A_h)$. $\tau_0$ can be decomposed as follows:
\begin{align}
    u(A)  - u(A_h) &\leq |u(A)  - u_h(A)| + |u_h(A_h)  - u(A_h)| + (u_h(A)  - u_h(A_h)) \\
    &\leq 2\varepsilon + u_h(A)  - u_h(A_h),
\end{align}
\noindent where we have applied Thm. $\ref{theorem:error_estimate}$ to get the discretization error estimate for $A$ and $A_h$. It is now enough to note that $u_h(A_h) > u_h(A)$ for small enough $h$ as $A_h$ is an estimated singularity (hence locally maximal arrival time) in the FE solution. Therefore, $\tau_0 < 2\varepsilon$, implying $(\Sigma_h \cup \Gamma_h) \subset (\Sigma \cup \Gamma)_{2\varepsilon} $. The final result follows by noting that the latter statement can be mirrored. That is, saying that $\Sigma_h \cup \Gamma_h$ is contained in a $2\varepsilon$ trust region around $\Sigma \cup \Gamma$ is equivalent to saying that $\Sigma \cup \Gamma$ is contained in a $2\varepsilon$ trust region around $\Sigma_h \cup \Gamma_h$.

\end{proof}
\label{theorem:cut_loci_error}
\end{theorem}

\begin{remark1}
    Note that even though the discretization error estimate from Thm. $\ref{theorem:error_estimate}$ does not hold explicitly on $\Gamma_\tau$, the chances of $A, A_h \in \Gamma_\tau$ are vanishingly small. Pignotti \cite{PIGNOTTI2002681} shows that $\Gamma$ is effectively a zero-dimensional subset in $\mathbb{R}^2$.
\end{remark1}

Therefore, Thm. $\ref{theorem:cut_loci_error}$ yields a direct practical tool for assessing whether the destination lies outside the region $(\Sigma_h \cup \Gamma_h) \subset (\Sigma \cup \Gamma)_{2\varepsilon}$ containing all singularities, and thus providing an approximation of the unique globally optimal trajectory.

\section{Numerical Results} \label{sec:results}
In this section we will numerically verify the linear convergence of our method and validate the discretization error bound from the previous chapter. In addition, we will demonstrate a practical application of the $2\varepsilon$ trust region around the approximate cut loci $\Sigma_h \cup \Gamma_h$. 

We consider a domain $\Omega_0=(0,1)^2 \subset \mathbb{R}^2$ with the origin fixed at $x_0 = [0.0, 0.5]$.
The airspeed is $\overline{v} = 1$ and the maximum wind magnitude $\overline{c}_0 = 0.5$. The boundary condition $g : \partial \mathcal{K} \to \mathbb{R}$ is defined on $\mathcal{B}_{0.1}(x_0)$ and computed using an efficient optimal control solver, cf.~\cite{discopter}.

We discretize the domain using a uniform distribution of points as a base for a subsequent Delaunay triangulation. The FE solution $u_h$ $\eqref{eq:FE_update_formula}$ satisfying $\Lambda_h u_h = u_h$ is computed using fixed point Jacobi iteration \cite{Bornemann2006,Rasch2007}. For the wind field $w$ we set up multiple test cases of increasing complexity:

a) A constant wind field with $w = \overline{c}_0 \times [1, -2]/\sqrt{5}$.

b) An everywhere smooth Gaussian wind vortex with a clockwise spin $s = -1$ (or counterclockwise $s = +1$):
\begin{align}
    w(x) = \left[ 
    \begin{aligned}
        - & s \overline{w}(r)   \sin(\alpha) \\
         & s  \overline{w}(r) \cos(\alpha)
    \end{aligned}
    \right],
\end{align}
where $\alpha = \arctan{\frac{(\textbf{x} - \textbf{z})_2}{(\textbf{x} - \textbf{z})_1}}$, and $R$ and $z$ are radius and vortex center, respectively. The radial basis function $\overline{w}(r)$ for $r = \lVert x - R\rVert $ is defined as follows:
\begin{equation}
	\overline{w}(r) = \gamma \frac{\beta r}{R}\exp \left(- \frac{\beta^2 r^2}{R^2} \right),
\end{equation}
where $\beta = 3, \gamma = \sqrt{2 e}.$ Here $R = 1/3$ and $z = [0.5, 0.5]$. The Jacobian of the wind field in this case is then bounded by $\overline{c}_1 = \beta \gamma$.

c) An array of 15 Gaussian wind vortices, arranged as in Fig. \ref{fig:cut_loci_trust_region}.

d) An array of 70 Gaussian wind vortices (similar to case (c)).

\noindent For each of the test problems (a) -- (d), we investigate the FE discretization error and demonstrate linear convergence of the algorithm. To eliminate the influence of the truncation error, we select the respective relative tolerance to be $\text{tol} = 10^{-10}$. We compare to a reference FE solution computed on a $1001 \times 1001$ grid. Fig.~$\ref{fig:convergence}$ shows that the $\lVert u - u_h \rVert _{\infty}$ error is of order $\mathcal{O}(h)$ for all test problems (a)--(d), as expected. Finally, we display the respective discretization error estimates computed using Thm.~\ref{theorem:error_estimate} on the corresponding $\Omega_\tau$. 

In addition, for test problem (c) we compute the temporal $2\varepsilon$ trust region around $\Sigma_h \cup \Gamma_h$. We observe that the aforementioned trust region encompasses the true cut locus computed using a $1001 \times 1001$ resolution, see Fig. $\ref{fig:cut_loci_trust_region}$. As expected, $\Sigma_h \cup \Gamma_h$ experiences spatial displacement due to error propagation, see Fig. $\ref{fig:cut_loci_trust_region}$(b). Nevertheless, the $2\varepsilon$ trust region captures this displacement (up to an accuracy of a single simplex, since the accuracy of $\Sigma_h \cup \Gamma_h$ is as such). Therefore, if the desired destination is outside the $2\varepsilon$ trust region, we can deduce that the candidate trajectory of the corresponding HJB equation $\eqref{chapter2:hjb_pde}$ is convergent to a unique globally optimal trajectory. 

\begin{figure}[H]
    \centering        
    \includegraphics[width=0.8\textwidth, trim={0cm 0cm 1cm 1.4cm}, clip]{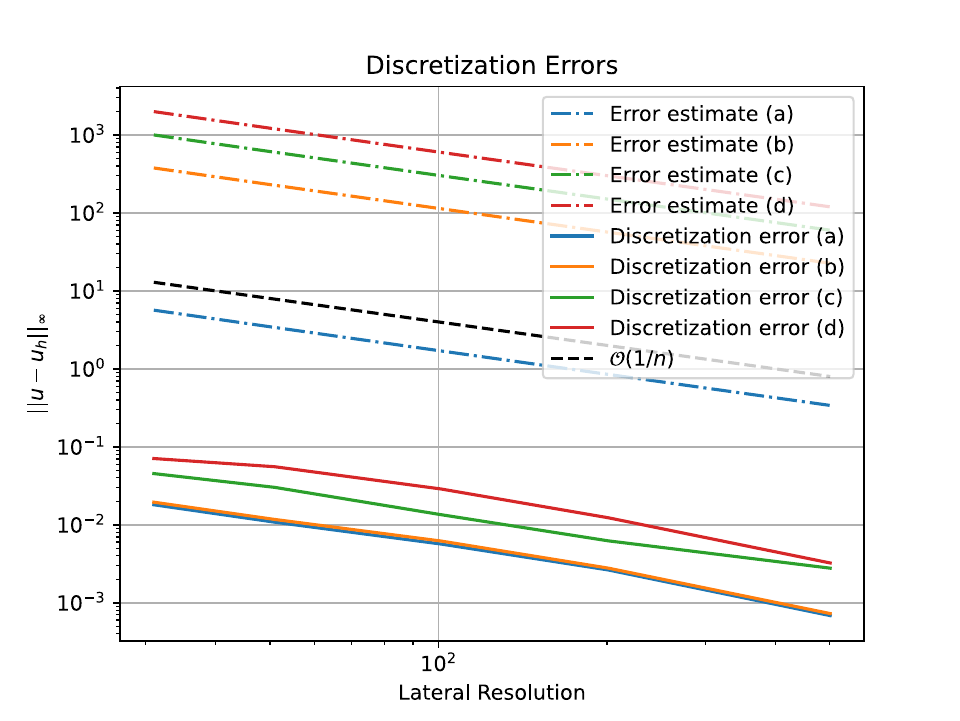}
    \caption{The discretization errors are displayed against their estimates for cases (a)--(d) with varying FE mesh size}
    \label{fig:convergence}
\end{figure}

\begin{figure}[H]
    \centering
    \begin{subfigure}{0.48\textwidth}
        \centering
        \includegraphics[width=\linewidth, trim={0.93cm 0.93cm 0.93cm 0.93cm}, clip]{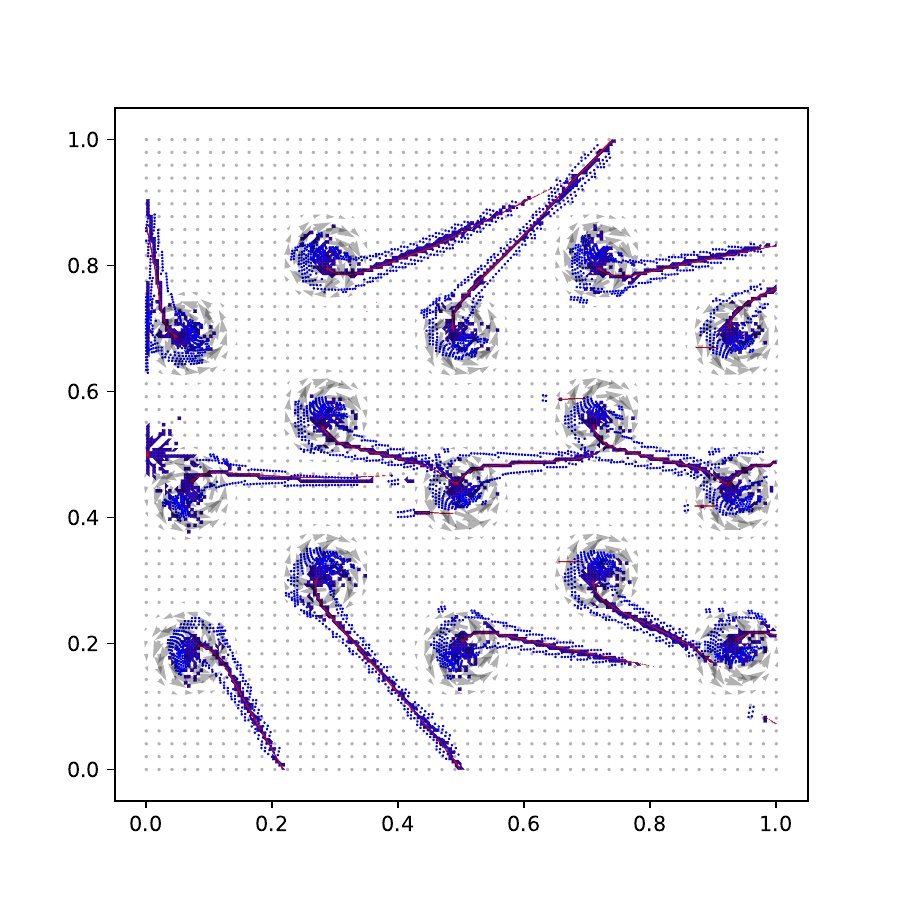}
        \caption{}
    \end{subfigure}
    \hfill
    \begin{subfigure}{0.48\textwidth}
        \centering
        \includegraphics[width=\linewidth, trim={0.93cm 0.93cm 0.93cm 0.93cm}, clip]{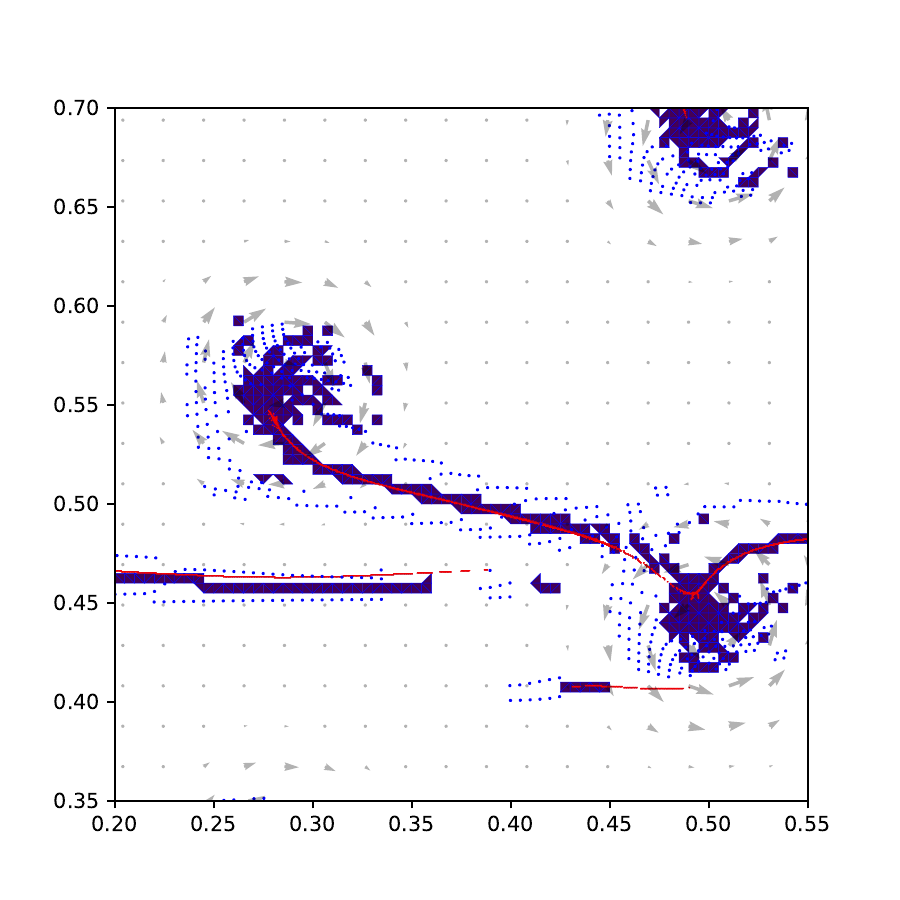}
        \caption{}
    \end{subfigure}
    \caption{Visualization of the temporal $2 \varepsilon$ trust region (blue points) around $\Sigma_h \cup \Gamma_h$ (blue triangles) computed using an FE solution with a $201 \times 201$ resolution. The FE solution with a $1001 \times 1001$ resolution is used as a reference solution to compute the a posteriori discretization error $\varepsilon$ and the cut loci $\Sigma \cup \Gamma$ (red triangles). The right plot is a zoomed-in version of the left plot depicting a positional shift in $\Sigma_h \cup \Gamma_h$}
    \label{fig:cut_loci_trust_region}
\end{figure}

\section{Acknowledgments}

\noindent Funded by the Deutsche Forschungsgemeinschaft (DFG, German Research
Foundation) under Germany's Excellence Strategy – The Berlin Mathematics
Research Center MATH+ (EXC-2046/1, EXC-2046/2, project ID: 390685689).

\printbibliography

@book{Bellman1957,
	title        = {{Dynamic Programming}},
	author       = {Bellman, R.},
	year         = 1957,
	publisher    = {Princeton University Press},
}

@phdthesis{Blanco2023,
	title        = {{Optimization Algorithms for the Flight Planning Problem}},
	author       = {Blanco Sandoval, M.D.},
	year         = 2023,
	school       = {Freie Universität Berlin},
}

@Article{discopter,
AUTHOR = {Bornd{\"o}rfer, R. and Danecker, F. and Weiser, M.},
TITLE = {A Discrete-Continuous Algorithm for Free Flight Planning},
JOURNAL = {Algorithms},
VOLUME = {14},
YEAR = {2021},
NUMBER = {1},
ARTICLE-NUMBER = {4},
DOI = {10.3390/a14010004}
}

@Article{Bornemann2006,
author={Bornemann, F. and Rasch, C.},
title={Finite-element Discretization of Static Hamilton-Jacobi Equations based on a Local Variational Principle},
journal={Computing and Visualization in Science},
year={2006},
volume={9},
number={2},
pages={57--69},
doi={10.1007/s00791-006-0016-y}
}

@book{cannarsa2004,
  title={Semiconcave Functions, Hamilton-Jacobi Equations, and Optimal Control},
  author={Cannarsa, P. and Sinestrari, C.},
  series={Progress in Nonlinear Differential Equations and Their Applications},
  year={2004},
  publisher={Springer}
}

@article{crandall1983,
 author = {M.G. Crandall and P.-L. Lions},
 journal = {Transactions of the American Mathematical Society},
 number = {1},
 pages = {1--42},
 publisher = {American Mathematical Society},
 title = {Viscosity Solutions of Hamilton-Jacobi Equations},
 volume = {277},
 year = {1983}
}

@phdthesis{Danecker_thesis,
	title        = {{A Discrete-Continuous Algorithm
for Globally Optimal Free Flight Trajectory Optimization}},
	author       = {Danecker, Fabian},
	year         = 2023,
	school       = {Freie Universit{\"a}t Berlin}
}

@misc{RAD,
	title        = {{Road Availability Document (RAD)}},
	author       = {{EUROCONTROL}},
	url          = {https://www.nm.eurocontrol.int/RAD/},
	urldate      = {2026-02-16},
}

@misc{Eurostat2025,
	title        = {{Air passenger transport by reporting country}},
	author       = {{Eurostat}},
	year         = 2025,
	url          = {https://ec.europa.eu/eurostat/databrowser/view/avia_paoc/default/table?lang=en&category=avia.avia_pa.avia_pao},
	urldate      = {2026-02-05},
}

@book{evans10,
  address = {Providence, R.I.},
  author = {Evans, L.C.},
  publisher = {American Mathematical Society},
  title = {Partial differential equations},
  year = 2010
}

@inproceedings{GirardetEtAl2014,
	title        = {Wind-Optimal Path Planning: Application to Aircraft Trajectories},
	author       = {B. Girardet and L. Lapasset and D. Delahaye and C. Rabut},
	year         = 2014,
	booktitle    = {13th International Conference on Control Automation Robotics Vision (ICARCV)},
	pages        = {1403--1408},
	doi          = {10.1109/ICARCV.2014.7064521}
}

@online{iata_passengers_2024,
  author       = {{International Air Transport Association}},
  title        = {Global Air Passenger Demand Reaches Record High in 2024},
  year         = {2025},
  url          = {https://www.iata.org/en/pressroom/2025-releases/2025-01-30-01/},
  note         = {Reports a total of about 4.89 billion air passengers worldwide in 2024},
}

@online{iata_outlook_2024,
  author       = {{International Air Transport Association}},
  title        = {Global Outlook for Air Transport -- June 2024 Update},
  year         = {2024},
  url          = {https://www.iata.org/en/pressroom/2025-releases/2025-06-02-01/},
  note         = {Reports ~99 billion gallons of jet fuel consumption in 2024}
}

@online{irena_aviation_transport,
  author       = {{International Renewable Energy Agency}},
  title        = {Decarbonising Hard-to-Abate Sectors: Aviation},
  year         = {2023},
  url          = {https://www.irena.org/Decarbonising-hard-to-abate-sectors-with-renewables-Enablers-and-recommendations/Transport-sector/Aviation},
  note         = {Aviation accounts for ~10\% of global transport energy use (IEA data synthesis)}
}

@Inbook{Mennucci2000,
author="Mennucci, Andrea C. G.",
editor="Djaferis, Theodore E.
and Schick, Irvin C.",
title="Regularity of Solutions to Hamilton-Jacobi Equations",
bookTitle="System Theory: Modeling, Analysis and Control",
year="2000",
publisher="Springer US",
address="Boston, MA",
pages="63--74",
isbn="978-1-4615-5223-9",
doi="10.1007/978-1-4615-5223-9_5",
url="https://doi.org/10.1007/978-1-4615-5223-9_5"
}

@misc{mordor_intelligence,
	title        = {{Aviation Market Size \& Share Analysis - Growth Trends And Forecast (2025 - 2030)}},
	author       = {{Mordor Intelligence Research \& Advisory}},
	year         = 2025,
	url          = {https://www.mordorintelligence.com/industry-reports/aviation-market},
	howpublished = {Issue 4},
	urldate      = {2026-02-05},
}

@article{NgSridharGrabbe2014,
	title        = {Optimizing Aircraft Trajectories with Multiple Cruise Altitudes in the Presence of Winds},
	author       = {Ng, H.K. and Sridhar, B. and Grabbe, S.},
	year         = 2014,
	journal      = {Journal of Aerospace Information Systems},
	volume       = 11,
	number       = 1,
	pages        = {35--47},
	doi          = {10.2514/1.I010084}
}

@article{PIGNOTTI2002681,
title = {Rectifiability results for singular and conjugate points of optimal exit time problems},
journal = {Journal of Mathematical Analysis and Applications},
volume = {270},
number = {2},
pages = {681-708},
year = {2002},
doi = {https://doi.org/10.1016/S0022-247X(02)00110-5},
author = {C. Pignotti}
}

@book{RTCA1995,
	title        = {{Final Report of RTCA Task Force 3 Free Flight Implementation}},
	author       = {{Radio Technical Commission for Aeronautics}},
	year         = 1995,
	publisher    = {RTCA},
}

@phdthesis{Rasch2007,
	author = {Rasch, C.M.A.},
	title = {Numerical Discretization of Static Hamilton-Jacobi Equations on Triangular Meshes},
	year = {2007},
	school = {Technische Universität München},
	pages = {102}
}

@article{FMM,
	title        = {A Fast Marching Level Set Method for Monotonically Advancing Fronts},
	author       = {J.A. Sethian},
	year         = 1996,
	journal      = {Proceedings of the National Academy of Sciences of the United States of America},
	volume       = 93,
	number       = 4,
	pages        = {1591--1595}
}

@article{FMM2,
 author = {J.A. Sethian},
 journal = {SIAM Review},
 number = {2},
 pages = {199--235},
 publisher = {Society for Industrial and Applied Mathematics},
 title = {Fast Marching Methods},
 volume = {41},
 year = {1999}
}

@article{SethianOUM_first,
author = {J.A. Sethian  and A. Vladimirsky },
title = {Ordered upwind methods for static Hamilton–Jacobi equations},
journal = {Proceedings of the National Academy of Sciences},
volume = {98},
number = {20},
pages = {11069-11074},
year = {2001},
doi = {10.1073/pnas.201222998}
}

@article{Tsitsiklis1995,
	title        = {{Efficient Algorithms for Globally Optimal Trajectories}},
	author       = {Tsitsiklis, J.N.},
	year         = 1995,
	journal      = {IEEE Transactions on Automatic Control},
	volume       = 40,
	number       = 9,
	pages        = {1528--1538},
	doi          = {10.1109/9.412624},
}

@article{WellsEtAl2023,
title = {Minimising emissions from flights through realistic wind fields with varying aircraft weights},
journal = {Transportation Research Part D: Transport and Environment},
volume = {117},
pages = {103660},
year = {2023},
doi = {https://doi.org/10.1016/j.trd.2023.103660},
author = {C.A. Wells and P.D. Williams and N.K. Nichols and D. Kalise and I. Poll}
}

@article{Zermelo1930,
	title        = {{\"Uber die Navigation in der Luft als Problem der Variationsrechnung `On the Navigation in the Air as a Problem of the Calculus of Variations'}},
	author       = {Zermelo, E.},
	year         = 1930,
	journal      = {Jahresbericht der Deutschen Mathematiker-Vereinigung},
	volume       = 39,
	pages        = {44--48},
}

\appendix

\section{Appendix}

\begin{lemma} Consider the setting of Theorem~\ref{theorem:error_estimate} with the slowness function as $\overline{f}(\cdot, \cdot) := 1/f(\cdot, \cdot)$. Then for any fixed $p \in S_1$,  $\overline{f}(\cdot, p)$ is Lipschitz continuous with the corresponding Lipschitz constant independent of $p$:
\begin{equation}
 L = \overline{c}_1 \sqrt{\frac{\overline{v} }{(\overline{v} - \overline{c}_0)^{5}} }
\end{equation}
    \begin{proof}
         For a fixed fixed $p \in S_1$ we have:
         \begin{equation}
             \overline{f}(x, p) = \frac{-w(x)^Tp + \sqrt{\overline{v}^2 - \lVert w(x) \rVert^2 + (w(x)^Tp)^2 }}{\overline{v}^2 - \lVert w(x) \rVert^2}, \quad x \in \mathbb{R}^2.
         \end{equation}
         Since $\overline{v} > \overline{c}_0$, it follows that $\overline{f}(\cdot, p) \in C^1(\mathbb{R}^2, \mathbb{R})$. Therefore, we can find its derivative at each point $x \in \mathbb{R}^2$ explicitly:
    \begin{align}
        \overline{f}_{\partial x}(x, p) &= \frac{w_x(x)^Tp}{\overline{v}^2 - \lVert w(x) \rVert ^2} \left ( \frac{w(x)^Tp}{\sqrt{\overline{v}^2 - \lVert w(x) \rVert ^2 + (w(x)^Tp)^2}}  - 1 \right) + \\
        & \frac{w_x(x)^Tw(x)}{(\overline{v}^2 - \lVert w(x) \rVert ^2)^2} \left (  \frac{\overline{v}^2 - \lVert w(x) \rVert ^2}{\sqrt{\overline{v}^2 - \lVert w(x) \rVert ^2 + (w(x)^Tp)^2}}  + \right. \\
        &\left.\frac{2w(x)^Tp \left(w(x)^Tp - \sqrt{\overline{v}^2 - \lVert w(x) \rVert ^2 + (w(x)^Tp)^2} \right)}{\sqrt{\overline{v}^2 - \lVert w(x) \rVert ^2 + (w(x)^Tp)^2}}\right)
    \end{align}
    \noindent and upper bound it:
    \begin{align}
        &\lVert \overline{f}_{\partial x}(x, p) \rVert \\
        &\leq \lVert w_x(x) \rVert \left (\frac{\overline{v} + \lVert w(x) \rVert}{(\overline{v}^2 - \lVert w(x) \rVert ^2)^{3/2}} +  \frac{\lVert w(x) \rVert}{ (\overline{v} - \lVert w(x) \rVert )^{2} \sqrt{\overline{v}^2 - \lVert w(x) \rVert ^2} }\right) \\
        &=  \lVert w_x(x) \rVert  \left (\frac{\overline{v} }{(\overline{v} - \lVert w(x) \rVert )^{5/2} (\overline{v} + \lVert w(x) \rVert)^{1/2}} \right)
    \end{align}
    Therefore, the Lipschitz constant $L$ is given by:
    \begin{equation}
        L := \sup\limits_{x \in \mathbb{R}^2}  \lVert\overline{f}_{\partial x}(x, p)\rVert
        \leq \overline{c}_1 \sqrt{\frac{\overline{v} }{(\overline{v} - \overline{c}_0)^{5}} }.
    \end{equation}
    \end{proof}
    \label{lemma:wind_constant}
\end{lemma}

\end{document}